\documentclass[11pt]{amsart}
\usepackage{amsmath,amssymb,amsthm,mathtools}
\usepackage{hyperref}
\usepackage{array}
\usepackage{booktabs}
\usepackage{graphicx}

\newtheorem{theorem}{Theorem}[section]
\newtheorem{lemma}[theorem]{Lemma}
\newtheorem{proposition}[theorem]{Proposition}

\theoremstyle{remark}
\newtheorem{remark}[theorem]{Remark}

\DeclareMathOperator{\Res}{Res}
\DeclareMathOperator{\ord}{ord}
\DeclareMathOperator{\Span}{Span}
\DeclareMathOperator{\cst}{cst}

\newcommand{\Z}{\mathbf{Z}}
\newcommand{\Q}{\mathbf{Q}}
\newcommand{\C}{\mathbf{C}}

\newcommand{\dd}{\mathrm{d}}

\newcommand{\vp}{v_p}

\begin{document}

\title[Sym$^3$ mod $p^4$ supercongruence]{Order drop, Hecke descent, and a mod $p^4$ supercongruence\\
for symmetric-cube hypergeometric coefficients}

\author{Alex Shvets}
\address{Haifa, Israel}
\email{alex@shvets.io}

\begin{abstract}
Let
\[
{}_2F_1\!\left(\frac13,\frac13;1;27z\right)^3=\sum_{n\ge0}A_n z^n.
\]
We prove the universal supercongruence
\[
A(pm)\equiv A(m)\pmod{p^4}
\qquad (p\ge5\ \text{prime},\ m\ge1).
\]

As an independent result of interest, we show that the specialized Mao--Tian cubic recurrence at the
arithmetic specialization $(1/3,1/3,1)$ drops from order $3$ to order $2$.

The proof of the supercongruence proceeds through the modular identification, congruences for the
coefficients of $C$, Lagrange--B\"urmann, and the vanishing of $F_r$ modulo $p^4$ for all $r\ge1$.
First, the generating series
\[
F(t):=\sum_{n\ge0}(-1)^nA_nt^n
\]
admits the modular parametrization
\[
F(t(\tau))=\frac{\eta(\tau)^9}{\eta(3\tau)^3},
\qquad
t(\tau)=\frac{\eta(3\tau)^{12}}{\eta(\tau)^{12}},
\]
whose logarithmic derivative is
\[
C(q)=3E_{5,\chi_0,\chi_3}(q).
\]
Second, the coefficients of $C$ satisfy the exact congruences
\[
c_{mp^r}\equiv c_{mp^{r-1}}\pmod{p^{4r}}
\qquad (p\ge5,\ m,r\ge1),
\]
and Lagrange--B\"urmann gives
\[
B_m:=(-1)^mA_m=\cst_q\!\left(\frac{C(q)}{t(q)^m}\right),
\]
where $\cst_q$ denotes the constant term in the $q$-expansion.

Third, for every prime $p\ge5$ and every $r\ge1$, the exact defect
\[
\widetilde G_r:=T_p\!\left(\frac{C}{t^{rp}}\right)-\frac{C}{t^r}
\]
is a weakly holomorphic modular form of weight $5$ on $\Gamma_0(3)$ with character $\chi_3$ and poles
only at the cusp $\infty$. Using
\[
M_5(\Gamma_0(3),\chi_3)=\Span\{C,tC\},
\qquad
C|W_3=-27i\,tC,
\]
one obtains a cusp-adapted expansion
\[
\widetilde G_r=\beta_{-1}tC+\sum_{j\ge0}\alpha_j\frac{C}{t^j}.
\]
The coefficients with $j\ge r$ are already divisible by $p^4$ from the principal part at
$\infty$, while applying $W_3$ gives
\[
\beta_{-1}\equiv \alpha_j\equiv0\pmod{p^4}
\qquad (0\le j\le r-1).
\]
Hence
\[
F_r(q):=\Lambda_p\!\left(\frac{C(q)}{t(q)^{rp}}\right)-\frac{C(q)}{t(q)^r}\equiv0\pmod{p^4}
\qquad (r\ge1).
\]

Taking constant terms gives
\[
B_{mp}-B_m=\cst_q(F_m)\equiv0\pmod{p^4},
\]
and therefore $A(pm)\equiv A(m)\pmod{p^4}$.
\end{abstract}

\maketitle

\section{Introduction}

Let
\[
F(z):={}_2F_1\!\left(\frac13,\frac13;1;z\right),
\qquad
F(27z)^3=\sum_{n\ge0}A_n z^n.
\]
The sequence begins
\[
1,\ 9,\ 135,\ 2439,\ 48519,\ 1023759,\ 22478121,\ 507897945,\dots.
\]

A theorem of Mao and Tian gives, for general parameters $(a,b,c)$, a third-order linear recurrence for the
Maclaurin coefficients of ${}_2F_1(a,b;c;z)^3$ \cite{MT}. At the arithmetic specialization
$(a,b,c)=(1/3,1/3,1)$, corresponding to a CM modular parametrization, however, the rescaled recurrence for
$A_n$ factors in the Ore algebra and drops to order~$2$. This order drop is an independent result of
interest and is not used in the proof of Theorem~A.

The proof of Theorem~A begins with the modular input. Set
\[
B_n:=(-1)^nA_n,
\qquad
F(t):=\sum_{n\ge0}B_nt^n.
\]
We prove
\[
F(t(\tau))=\frac{\eta(\tau)^9}{\eta(3\tau)^3},
\qquad
t(\tau)=\frac{\eta(3\tau)^{12}}{\eta(\tau)^{12}},
\]
and identify
\[
C(q):=F(t(q))\,\frac{q}{t(q)}\frac{\dd t}{\dd q}
\]
with the Eisenstein series $3E_{5,\chi_0,\chi_3}(q)$. The coefficients of $C$ satisfy the congruences
\[
c_{mp^r}\equiv c_{mp^{r-1}}\pmod{p^{4r}}
\qquad (p\ge5,\ m,r\ge1),
\]
and Lagrange--B\"urmann yields
\[
B_m=\cst_q\!\left(\frac{C(q)}{t(q)^m}\right),
\qquad
H(q):=\frac{q}{t(q)},
\]
where $\cst_q$ denotes the constant term in the $q$-expansion.

The decisive step is a Hecke descent on weakly holomorphic forms. For every $r\ge1$ define
\[
F_r(q):=\Lambda_p\!\left(\frac{C(q)}{t(q)^{rp}}\right)-\frac{C(q)}{t(q)^r},
\qquad
\widetilde G_r:=T_p\!\left(\frac{C}{t^{rp}}\right)-\frac{C}{t^r}.
\]
The Hecke decomposition
\[
T_p=\Lambda_p+\chi_3(p)p^4V_p
\]
shows that $F_r\equiv \widetilde G_r\pmod{p^4}$. The space
\[
M_5(\Gamma_0(3),\chi_3)=\Span\{C,tC\}
\]
is two-dimensional, and the exact defect $\widetilde G_r$ has poles only at the cusp $\infty$. Hence it
admits a cusp-adapted expansion
\[
\widetilde G_r=\beta_{-1}tC+\sum_{j\ge0}\alpha_j\frac{C}{t^j}.
\]
The terms with $j\ge r$ are already $0\bmod p^4$ from the principal part at $\infty$, while after applying
$W_3$ one obtains
\[
\beta_{-1}\equiv \alpha_j\equiv0\pmod{p^4}
\qquad (0\le j\le r-1).
\]
This gives
\[
F_r(q)\equiv0\pmod{p^4}
\qquad (r\ge1),
\]
uniformly in $p$ and without any reduction to finitely many layers.

The main theorem is immediate.

\medskip
\noindent\textbf{Theorem A.}
For every prime $p\ge5$ and every integer $m\ge1$,
\[
A(pm)\equiv A(m)\pmod{p^4}.
\]

Indeed,
\[
B_{mp}-B_m
=
\cst_q\!\left(
\Lambda_p\!\left(\frac{C(q)}{t(q)^{mp}}\right)-\frac{C(q)}{t(q)^m}
\right)
=
\cst_q(F_m)
\equiv0\pmod{p^4}.
\]

The paper is organized as follows. Section~\ref{sec:orderdrop} proves the independent order drop.
Sections~\ref{sec:modular}--\ref{sec:hecke-fricke} contain the proof of Theorem~A: Section~\ref{sec:modular}
proves the modular identification. Section~\ref{sec:eisenstein} develops congruences for the coefficients of
$C$ and the Lagrange--B\"urmann formula. Section~\ref{sec:hecke-fricke} proves the Hecke descent and the
vanishing of $F_r$ modulo $p^4$ for all $r\ge1$, and deduces Theorem~A.

\section{Order drop at the specialization $(1/3,1/3,1)$}\label{sec:orderdrop}

We work in the Ore algebra
$\Q[n]\langle S\rangle$,
$Sf(n)=f(n+1)$,
and $S\,P(n)=P(n+1)S$ for $P\in\Q[n]$.

Set
\[
R(n):=18n^4+108n^3+250n^2+264n+107
\]
and
\[
L_2:=729(n+1)^4-3R(n)S+(n+2)^4S^2\in\Q[n]\langle S\rangle.
\]

\begin{theorem}\label{thm:orderdrop}
Let
\[
{}_2F_1\!\left(\frac13,\frac13;1;27z\right)^3=\sum_{n\ge0}A_n z^n.
\]
Then $A_n$ satisfies the second-order recurrence
\begin{equation}\label{eq:S1-rec}
(n+2)^4A_{n+2}
-3\bigl(18n^4+108n^3+250n^2+264n+107\bigr)A_{n+1}
+729(n+1)^4A_n=0
\end{equation}
for all $n\ge 0$, with initial values
$A_0=1$, $A_1=9$.

Moreover, the specialization at $(a,b,c)=(\frac13,\frac13,1)$ of the rescaled
Mao--Tian order-$3$ operator is
\begin{equation}\label{eq:L3explicit}
\begin{aligned}
L_3:={}&-19683(n+1)^4
+81\bigl(27n^4+180n^3+466n^2+552n+251\bigr)S \\
&\quad-3\bigl(27n^4+252n^3+898n^2+1448n+891\bigr)S^2
+(n+3)^4S^3,
\end{aligned}
\end{equation}
and it factors in the Ore algebra as
\begin{equation}\label{eq:ore-factor}
L_3=(S-27)L_2.
\end{equation}
In particular, the generic order-$3$ recurrence drops to order $2$ at this arithmetic specialization.
\end{theorem}

\begin{proof}
Let
\[
{}_2F_1\!\left(\frac13,\frac13;1;z\right)^3=\sum_{n\ge0}v_n z^n,
\]
so that $A_n=27^n v_n$.
By \cite[Theorem~3.1]{MT}, the sequence $v_n$ satisfies a third-order recurrence
\[
v_{n+1}=\beta_0(n)v_n+\beta_1(n)v_{n-1}+\beta_2(n)v_{n-2}
\qquad (n\ge 1).
\]
Specializing the explicit coefficients of \cite[Theorem~3.1]{MT} at
$(a,b,c)=(\frac13,\frac13,1)$ gives
\[
\beta_0(n)=\frac{27n^4+36n^3+34n^2+16n+3}{9(n+1)^4},
\]
\[
\beta_1(n)=-\frac{27n^4-36n^3+34n^2-16n+3}{9(n+1)^4},
\qquad
\beta_2(n)=\frac{(n-1)^4}{(n+1)^4}.
\]
Substituting $v_n=27^{-n}A_n$, shifting $n\mapsto n+2$, and clearing denominators yields
\[
(n+3)^4A_{n+3}
-3\bigl(27n^4+252n^3+898n^2+1448n+891\bigr)A_{n+2}
\]
\[
\qquad\qquad
+81\bigl(27n^4+180n^3+466n^2+552n+251\bigr)A_{n+1}
-19683(n+1)^4A_n=0,
\]
which is exactly the recurrence $L_3A=0$ with $L_3$ as in \eqref{eq:L3explicit}.

To factor $L_3$, compute in the Ore algebra:
\[
\begin{aligned}
(S-27)L_2
={}&S\bigl(729(n+1)^4\bigr)-3S\bigl(R(n)\bigr)S+S\bigl((n+2)^4\bigr)S^2 \\
&\quad-27\cdot 729(n+1)^4+81R(n)S-27(n+2)^4S^2 \\
={}&-19683(n+1)^4+\bigl(729(n+2)^4+81R(n)\bigr)S \\
&\quad-\bigl(3R(n+1)+27(n+2)^4\bigr)S^2+(n+3)^4S^3.
\end{aligned}
\]
Now
\[
729(n+2)^4+81R(n)
=81\bigl(27n^4+180n^3+466n^2+552n+251\bigr)
\]
and
\[
3R(n+1)+27(n+2)^4
=3\bigl(27n^4+252n^3+898n^2+1448n+891\bigr),
\]
so indeed $(S-27)L_2=L_3$.

Now set $w:=L_2A$.
Since $L_3A=0$, the factorization \eqref{eq:ore-factor} yields
$(S-27)w=0$,
that is,
\[
w_{n+1}=27w_n \qquad (n\ge 0).
\]
Therefore it is enough to check that $w_0=0$.
From the definition of $A_n$ one finds
$A_0=1$, $A_1=9$, $A_2=135$.
Hence
\[
w_0
=729A_0-3R(0)A_1+2^4A_2
=729-2889+2160=0.
\]
It follows that $w_n=0$ for all $n\ge 0$,
so $L_2A=0$, which is exactly \eqref{eq:S1-rec}.
\end{proof}

\begin{remark}
Once the factorization \eqref{eq:ore-factor} is known, the passage from order $3$ to order $2$ is immediate:
the residual sequence $w=L_2A$ satisfies
$w_{n+1}=27w_n$, and one initial check kills it.
\end{remark}

\begin{remark}
Although the derivation above invokes \cite[Theorem~3.1]{MT} to obtain the general form of the
cubic recurrence, the specific specialized coefficients used here are independently verifiable
by direct substitution: the resulting recurrence \eqref{eq:S1-rec} annihilates the explicit
sequence $A_0=1,\,A_1=9,\,A_2=135,\,A_3=2439,\,\ldots$ derived from the hypergeometric expansion,
as is implicit in the initial-value check $w_0=0$ in the proof above.
\end{remark}

\section{Modular identification}\label{sec:modular}

Let
\[
q=e^{2\pi i\tau},
\qquad
t(\tau):=\frac{\eta(3\tau)^{12}}{\eta(\tau)^{12}},
\qquad
B_n:=(-1)^nA_n.
\]
Then
\[
\sum_{n\ge 0} B_n t^n
={}_2F_1\!\left(\frac13,\frac13;1;-27t\right)^3.
\]
Define
\[
C(q):=\Bigl(\sum_{n\ge 0} B_n t^n\Bigr)\cdot\frac{q}{t}\cdot\frac{\dd t}{\dd q}.
\]

\begin{theorem}\label{thm:modular}
With the notation above,
\begin{equation}\label{eq:mod-id}
\sum_{n\ge 0} B_n\, t(\tau)^n
=
\frac{\eta(\tau)^9}{\eta(3\tau)^3}.
\end{equation}
Consequently,
\begin{equation}\label{eq:C-eisenstein}
C(q)=3E_{5,\chi_0,\chi_3}(q),
\end{equation}
where $\chi_3(\cdot)=\left(\frac{\cdot}{3}\right)$.
If $C(q)=1+\sum_{n\ge 1} c_n q^n$, then
\begin{equation}\label{eq:c-divisor-sum}
c_n = 3\,\sigma_{4,\chi_3}(n),
\qquad
\sigma_{4,\chi_3}(n):=\sum_{d\mid n}\chi_3(d)\,d^4.
\end{equation}
\end{theorem}

\begin{proof}
Let
\[
a(q):=\sum_{m,n\in\Z}q^{m^2+mn+n^2},
\qquad
b(q):=\frac{\eta(\tau)^3}{\eta(3\tau)},
\qquad
c(q):=3\frac{\eta(3\tau)^3}{\eta(\tau)}.
\]
The cubic theory of Borwein--Borwein--Garvan gives
\[
a(q)^3=b(q)^3+c(q)^3
\qquad\text{and}\qquad
{}_2F_1\!\left(\frac13,\frac23;1;\frac{c(q)^3}{a(q)^3}\right)=a(q)
\]
\cite[Theorem~2.3 and Corollary~2.4]{BBG}.
Since
\[
-27t(\tau)=-\frac{c(q)^3}{b(q)^3},
\]
Pfaff's transformation
\[
{}_2F_1(a,b;c;z)=(1-z)^{-a}{}_2F_1\!\left(a,c-b;c;\frac{z}{z-1}\right)
\]
with $a=b=\frac13$, $c=1$, and $z=-27t(\tau)$ yields
\[
{}_2F_1\!\left(\frac13,\frac13;1;-27t(\tau)\right)
=(1-z)^{-1/3}{}_2F_1\!\left(\frac13,\frac23;1;\frac{z}{z-1}\right).
\]
Now
\[
\frac{z}{z-1}
=
\frac{c(q)^3}{b(q)^3+c(q)^3}
=
\frac{c(q)^3}{a(q)^3},
\]
and
\[
(1-z)^{-1/3}
=
\left(1+\frac{c(q)^3}{b(q)^3}\right)^{-1/3}
=
\left(\frac{a(q)^3}{b(q)^3}\right)^{-1/3}
=
\frac{b(q)}{a(q)}.
\]
Here $(1-z)^{-1/3}$ denotes the analytic branch near $z=0$ normalized by
$(1-z)^{-1/3}\big|_{z=0}=1$; both sides of the equality above are analytic at $q=0$ with constant term
$1$, so they coincide identically as power series in $q$.
Therefore
\[
{}_2F_1\!\left(\frac13,\frac13;1;-27t(\tau)\right)
=
\frac{b(q)}{a(q)}
{}_2F_1\!\left(\frac13,\frac23;1;\frac{c(q)^3}{a(q)^3}\right)
=
\frac{b(q)}{a(q)}a(q)
=
\frac{\eta(\tau)^3}{\eta(3\tau)}.
\]
Cubing gives \eqref{eq:mod-id}.

For the differential statement, \cite[Example~5.2]{Moy} gives
\[
\frac{\eta(\tau)^9}{\eta(3\tau)^3}\cdot\frac{q}{t}\cdot\frac{\dd t}{\dd q}
=3E_{5,\chi_0,\chi_3}(\tau),
\]
which is exactly \eqref{eq:C-eisenstein}. We use the Fourier expansion
\[
E_{5,\chi_0,\chi_3}(q)
=
\frac13+\sum_{n\ge1}\left(\sum_{d\mid n}\chi_3(d)d^4\right)q^n.
\]
Hence
\[
C(q)=1+\sum_{n\ge1}3\left(\sum_{d\mid n}\chi_3(d)d^4\right)q^n,
\]
so \eqref{eq:c-divisor-sum} follows.
\end{proof}

\begin{remark}\label{rem:moy-precedence}
The eta-product $\eta(\tau)^9/\eta(3\tau)^3$, the Hauptmodul $t$, and the Eisenstein identification
$C(q)=3E_{5,\chi_0,\chi_3}$ already appear in \cite[Example~5.2]{Moy}. The derivation of the
generating-series identity~\eqref{eq:mod-id} via Pfaff's transformation and the Borwein cubic theory
provides a convenient route from the hypergeometric series to the modular parametrization.
\end{remark}

\section{Congruences for the coefficients of $C$ and Lagrange--B\"urmann}\label{sec:eisenstein}

Define
\[
C(q)=1+\sum_{n\ge1} c_n q^n,
\qquad
c_n=3\sigma_{4,\chi_3}(n),
\qquad
\sigma_{4,\chi_3}(n):=\sum_{d\mid n}\chi_3(d)d^4.
\]
Also define
\[
H(q):=\frac{q}{t(q)}=\prod_{\substack{n\ge1\\3\nmid n}}(1-q^n)^{12}\in 1+q\Z[[q]].
\]

\subsection*{4.1. Euler factors}

\begin{lemma}\label{lem:euler-factor}
The arithmetic function $\sigma_{4,\chi_3}$ is multiplicative. For every prime $p\neq 3$ and every integer $N\ge0$ one has
\[
\sigma_{4,\chi_3}(p^N)=1+\chi_3(p)p^4+\chi_3(p)^2p^8+\cdots+\chi_3(p)^N p^{4N}.
\]
Consequently, if $m=p^a m_0$ with $a\ge0$ and $p\nmid m_0$, then for every $r\ge1$,
\[
\sigma_{4,\chi_3}(mp^r)-\sigma_{4,\chi_3}(mp^{r-1})
=
\chi_3(p)^{a+r}p^{4(a+r)}\sigma_{4,\chi_3}(m_0).
\]
\end{lemma}

\begin{proof}
Since $n\mapsto \chi_3(n)n^4$ is completely multiplicative on integers prime to $3$, its divisor sum $\sigma_{4,\chi_3}$ is multiplicative. For $p\neq3$,
\[
\sigma_{4,\chi_3}(p^N)=\sum_{j=0}^N \chi_3(p^j)p^{4j}=
\sum_{j=0}^N \chi_3(p)^j p^{4j},
\]
which is the displayed Euler factor.

Now write $m=p^am_0$ with $p\nmid m_0$. Multiplicativity gives
\begin{align*}
\sigma_{4,\chi_3}(mp^r)
&=\sigma_{4,\chi_3}(p^{a+r})\sigma_{4,\chi_3}(m_0),\\
\sigma_{4,\chi_3}(mp^{r-1})
&=\sigma_{4,\chi_3}(p^{a+r-1})\sigma_{4,\chi_3}(m_0).
\end{align*}
Subtracting and using the Euler factor formula yields the claim.
\end{proof}

\subsection*{4.2. Congruences for the coefficients of $C$}

\begin{theorem}[Congruences for the coefficients of $C$]\label{thm:tower}
For every prime $p\ge5$ and all integers $m,r\ge1$,
\[
c_{mp^r}\equiv c_{mp^{r-1}}\pmod{p^{4r}}.
\]
In particular,
\[
\Lambda_p(C)(q)=1+\sum_{n\ge1}c_{np}q^n\equiv C(q)\pmod{p^4}.
\]
\end{theorem}

\begin{proof}
Since $c_n=3\sigma_{4,\chi_3}(n)$ and $p\neq3$, Lemma~\ref{lem:euler-factor} gives
\[
c_{mp^r}-c_{mp^{r-1}}=3\chi_3(p)^{a+r}p^{4(a+r)}\sigma_{4,\chi_3}(m_0)
\]
when $m=p^am_0$ with $p\nmid m_0$. Therefore
\[
\vp\bigl(c_{mp^r}-c_{mp^{r-1}}\bigr)\ge 4(a+r)\ge 4r,
\]
which is exactly the congruence.

The statement about $\Lambda_p(C)$ is the case $r=1$ applied coefficientwise.
\end{proof}

\subsection*{4.3. Lagrange--B\"urmann}

\begin{theorem}[Lagrange--B\"urmann coefficient formula]\label{thm:LB}
For every $m\ge0$,
\[
B_m=\cst_q\!\left(\frac{C(q)}{t(q)^m}\right).
\]
\end{theorem}

\begin{proof}
Since
\[
F(t)=\sum_{n\ge0}B_nt^n,
\]
we have
\[
B_m=\cst_t\!\left(\frac{F(t)}{t^m}\right)
=\Res_{t=0}\left(\frac{F(t)}{t^m}\frac{\dd t}{t}\right).
\]
From the $\eta$-product definition,
\[
t(q)=\frac{\eta(3\tau)^{12}}{\eta(\tau)^{12}}=q+12q^2+90q^3+\cdots\in q\Z[[q]],
\qquad [q]t(q)=1,
\]
so $q\mapsto t(q)$ is a formal change of coordinates at the origin, with inverse
$q(t)\in t\Z[[t]]$. The substitution below is therefore legitimate as a formal
change of variables in Laurent series: by the residue change-of-variables formula,
\[
\Res_{t=0}\left(\frac{F(t)}{t^m}\frac{\dd t}{t}\right)
=
\Res_{q=0}\left(
\frac{F(t(q))}{t(q)^m}\frac{q\,t'(q)}{t(q)}\frac{\dd q}{q}
\right).
\]
Now
\[
F(t(q))\,\frac{q}{t(q)}\,\frac{\dd t}{\dd q}=C(q),
\]
so
\[
B_m
=
\Res_{q=0}\left(\frac{C(q)}{t(q)^m}\frac{\dd q}{q}\right)
=
\cst_q\!\left(\frac{C(q)}{t(q)^m}\right).
\]
\end{proof}

\section{Hecke descent and the vanishing of $F_r$ modulo $p^4$ for all $r\ge1$}\label{sec:hecke-fricke}

Throughout this section $p\ge5$ is prime. For
\[
f(q)=\sum_{n\gg-\infty} a_n q^n
\]
define
\[
\Lambda_p(f):=\sum_{n\gg-\infty} a_{np}q^n,
\qquad
V_p(f):=f(q^p)=\sum_{n\gg-\infty} a_n q^{pn}.
\]

\subsection*{5.1. Hecke defects}

\begin{lemma}\label{lem:hecke-decomp}
Let $f(q)=\sum_{n\gg-\infty} a_n q^n\in M_k^!(\Gamma_0(3),\chi_3)$ and let $p\ge5$ be prime. Then
\[
T_p f = \Lambda_p(f)+\chi_3(p)p^{k-1}V_p(f).
\]
In particular, for weight $k=5$,
\[
T_p=\Lambda_p+\chi_3(p)p^4V_p.
\]
\end{lemma}

\begin{proof}
For $p\nmid3$, the usual Hecke operator on $M_k^!(\Gamma_0(3),\chi_3)$ is given by
\[
T_p f(\tau)=\chi_3(p)p^{k-1}f(p\tau)+\frac1p\sum_{b=0}^{p-1} f\!\left(\frac{\tau+b}{p}\right).
\]
The double-coset definition shows that $T_p$ preserves both the level and the nebentypus.

Now expand $f$ as a Laurent series. The first term is
\[
\chi_3(p)p^{k-1}f(p\tau)=\chi_3(p)p^{k-1}\sum_{n\gg-\infty} a_n q^{pn}
=
\chi_3(p)p^{k-1}V_p(f).
\]
For the second term, write
\[
f\!\left(\frac{\tau+b}{p}\right)=\sum_{n\gg-\infty} a_n \zeta_p^{bn}q^{n/p}.
\]
Summing over $b$ gives
\[
\frac1p\sum_{b=0}^{p-1}f\!\left(\frac{\tau+b}{p}\right)
=\sum_{n\gg-\infty} a_n\left(\frac1p\sum_{b=0}^{p-1}\zeta_p^{bn}\right)q^{n/p}
=\sum_{m\gg-\infty} a_{pm}q^m
=\Lambda_p(f),
\]
because
\[
\frac1p\sum_{b=0}^{p-1}\zeta_p^{bn}=
\begin{cases}
1,& p\mid n,\\
0,& p\nmid n.
\end{cases}
\]
This proves the formula.
\end{proof}

For each integer $r\ge1$ set
\[
F_r(q):=\Lambda_p\!\left(\frac{C(q)}{t(q)^{rp}}\right)-\frac{C(q)}{t(q)^r},
\qquad
\widetilde G_r:=T_p\!\left(\frac{C}{t^{rp}}\right)-\frac{C}{t^r}.
\]
For later use, note that
\[
\frac{C(q)}{t(q)^m}=q^{-m}C(q)H(q)^m\in q^{-m}\Z[[q]]
\qquad (m\ge0),
\]
since $C(q)\in\Z[[q]]$ and $H(q)=q/t(q)\in 1+q\Z[[q]]$. In particular, all coefficientwise
congruences below refer to $\Z$-valued coefficients of explicit Laurent series.

\begin{proposition}\label{prop:defect-congruence}
For every $r\ge1$ one has
\[
F_r(q)\equiv \widetilde G_r(q)\pmod{p^4}.
\]
Consequently, $F_r\bmod p^4$ is represented by a weakly holomorphic modular form of weight $5$ and character $\chi_3$ on $\Gamma_0(3)$.
\end{proposition}

\begin{proof}
Apply Lemma~\ref{lem:hecke-decomp} with $k=5$ to the weakly holomorphic modular form $C/t^{rp}$:
\[
T_p\!\left(\frac{C}{t^{rp}}\right)
=
\Lambda_p\!\left(\frac{C}{t^{rp}}\right)+\chi_3(p)p^4V_p\!\left(\frac{C}{t^{rp}}\right).
\]
Subtract $C/t^r$ from both sides. The second term on the right is coefficientwise divisible by $p^4$, so the stated congruence follows.
\end{proof}

\begin{proposition}\label{prop:defect-pole-order}
For every $r\ge1$, the defect $F_r(q)$ has pole order at most $r-1$ at the cusp $\infty$:
\[
F_r(q)=O(q^{-r+1}).
\]
\end{proposition}

\begin{proof}
Since $t(q)=q+O(q^2)$ and $H(q)=q/t(q)\in1+q\Z[[q]]$, we have
\[
\frac{C(q)}{t(q)^{rp}}=q^{-rp}C(q)H(q)^{rp}=q^{-rp}(1+O(q)).
\]
Applying $\Lambda_p$ gives
\[
\Lambda_p\!\left(\frac{C(q)}{t(q)^{rp}}\right)=q^{-r}(1+O(q)).
\]
On the other hand,
\[
\frac{C(q)}{t(q)^r}=q^{-r}C(q)H(q)^r=q^{-r}(1+O(q)).
\]
The principal coefficient $q^{-r}$ therefore cancels, and $F_r(q)=O(q^{-r+1})$.
\end{proof}

\subsection*{5.2. Fricke identities and the weight-$5$ space on $\Gamma_0(3)$}

Write
\[
\Theta(\tau):=\frac{\eta(\tau)^9}{\eta(3\tau)^3},
\qquad
D:=q\frac{\dd}{\dd q}=\frac{1}{2\pi i}\frac{\dd}{\dd\tau},
\]
so that
\[
C(\tau)=\Theta(\tau)\,\frac{Dt(\tau)}{t(\tau)}.
\]

\begin{proposition}\label{prop:fricke-C}
Let
\[
w_3:=\begin{pmatrix}0&-1\\ 3&0\end{pmatrix},
\qquad
W_3f:=f|_k w_3.
\]
Then
\[
t|_0W_3=\frac{3^{-6}}{t},
\qquad
\Theta|_3W_3=27i\,t\Theta,
\qquad
\left(\frac{Dt}{t}\right)\!\Big|_2W_3=-\frac{Dt}{t}.
\]
Consequently,
\[
C|W_3=-27i\,tC.
\]
Since $W_3^2=-\mathrm{id}$ on weight $5$, one also has
\[
(tC)|W_3=-\frac{i}{27}C.
\]
Therefore
\[
\ord_0(C)=1,
\qquad
\ord_0(tC)=0.
\]
In particular, for every integer $j\ge0$,
\[
\ord_0\!\left(\frac{C}{t^j}\right)=j+1.
\]
\end{proposition}

\begin{proof}
The identity $t|_0W_3=3^{-6}/t$ is equivalent to
\[
t\!\left(-\frac{1}{3\tau}\right)=\frac{3^{-6}}{t(\tau)},
\]
which follows immediately from the eta-quotient definition of $t$ and the transformation formula
\[
\eta\!\left(-\frac1\tau\right)=(-i\tau)^{1/2}\eta(\tau).
\]

For $\Theta$, the weight-$3$ slash formula gives
\[
(\Theta|_3W_3)(\tau)
=
3^{3/2}(3\tau)^{-3}\Theta\!\left(-\frac{1}{3\tau}\right).
\]
Using the eta transformation in the numerator and denominator,
\[
\Theta\!\left(-\frac{1}{3\tau}\right)
=
\frac{\eta(-1/(3\tau))^9}{\eta(-1/\tau)^3}
=
3^{9/2}(-i\tau)^3\frac{\eta(3\tau)^9}{\eta(\tau)^3},
\]
and therefore
\[
\Theta|_3W_3
=
27i\,\frac{\eta(3\tau)^9}{\eta(\tau)^3}
=
27i\,t\Theta.
\]

We now compute $(Dt/t)|_2W_3$. Applying the chain rule to the explicit identity
$t|_0W_3=3^{-6}/t$ gives
\[
D(t|_0W_3)=D\!\left(\frac{3^{-6}}{t}\right)=-\frac{3^{-6}}{t^2}\,Dt.
\]
On the other hand, for any meromorphic function $f$ of weight $0$ on $\mathfrak H$, the
derivative $Df$ is of weight $2$ in the sense that the chain rule gives
\[
D(f|_0\alpha)=(Df)|_2\alpha
\qquad (\alpha\in GL_2^+(\Q)),
\]
and applying this with $f=t$ and $\alpha=w_3$ yields $(Dt)|_2W_3=D(t|_0W_3)=-3^{-6}t^{-2}Dt$.
Dividing by $t|_0W_3=3^{-6}/t$ gives
\[
\left(\frac{Dt}{t}\right)\!\Big|_2W_3
=
\frac{(Dt)|_2W_3}{t|_0W_3}
=
\frac{-3^{-6}t^{-2}Dt}{3^{-6}t^{-1}}
=
-\frac{Dt}{t}.
\]
Since $C=\Theta\cdot Dt/t$, the product formula for slash operators gives
\[
C|W_3
=
(\Theta|_3W_3)\bigl((Dt/t)|_2W_3\bigr)
=
(27it\Theta)\!\left(-\frac{Dt}{t}\right)
=
-27i\,tC.
\]

Now
\[
w_3^2=-3I,
\]
so for every weight-$5$ form $f$ one has
\[
(f|W_3)|W_3=f|_5(-3I)=-f.
\]
Applying $W_3$ to the identity $C|W_3=-27i\,tC$ gives
\[
-C=-27i\,(tC)|W_3,
\]
hence
\[
(tC)|W_3=-\frac{i}{27}C.
\]
Because $\ord_\infty(C)=0$ and $\ord_\infty(tC)=1$, this yields
\[
\ord_0(C)=\ord_\infty(C|W_3)=1,
\qquad
\ord_0(tC)=\ord_\infty((tC)|W_3)=0.
\]
Finally,
\[
\ord_0\!\left(\frac{C}{t^j}\right)=\ord_0(C)-j\,\ord_0(t)=1+j
\qquad (j\ge0),
\]
because $\ord_0(t)=-1$.
\end{proof}

\begin{proposition}\label{prop:weight5-basis}
One has
\[
\dim M_5(\Gamma_0(3),\chi_3)=2,
\qquad
M_5(\Gamma_0(3),\chi_3)=\Span\{C,tC\}.
\]
More precisely,
\[
C=3E_{5,\chi_0,\chi_3},
\qquad
tC=E_{5,\chi_3,\chi_0}.
\]
The two basis elements are linearly independent, since
\[
C(q)=1+3q-45q^2+3q^3+\cdots,
\qquad
tC(q)=q+15q^2+81q^3+\cdots.
\]
\end{proposition}

\begin{proof}
By Theorem~\ref{thm:modular}, the form $C=3E_{5,\chi_0,\chi_3}$ is holomorphic of weight $5$ on
$\Gamma_0(3)$ with character $\chi_3$. Since $t$ is a weight-$0$ modular function on $\Gamma_0(3)$ and
Proposition~\ref{prop:fricke-C} gives $\ord_0(tC)=0$, the form $tC$ is also holomorphic of weight $5$ with
character $\chi_3$. The displayed $q$-expansions show that $C$ and $tC$ are linearly independent.

The dimension $\dim M_5(\Gamma_0(3),\chi_3)=2$ follows from the standard dimension formula for spaces of
modular forms with character on $\Gamma_0(N)$ \cite[\S3.5 and \S3.6]{DS}; concretely, for odd weight~$5$
on $\Gamma_0(3)$ with the quadratic character $\chi_3$, the space of cusp forms $S_5(\Gamma_0(3),\chi_3)$
vanishes, while the Eisenstein subspace is spanned by the two Eisenstein series $E_{5,\chi_0,\chi_3}$ and
$E_{5,\chi_3,\chi_0}$ associated to the two decompositions $\chi_3=\chi_0\cdot\chi_3$ and
$\chi_3=\chi_3\cdot\chi_0$ of the character as a product of Dirichlet characters modulo divisors of $3$.
Hence $C$ and $tC$ form a basis of this space.
The identity $C=3E_{5,\chi_0,\chi_3}$ is
Theorem~\ref{thm:modular}. The form $E_{5,\chi_3,\chi_0}$ has zero constant term and Fourier expansion
\[
E_{5,\chi_3,\chi_0}(q)
=
\sum_{n\ge1}\left(\sum_{d\mid n}\chi_3(n/d)d^4\right)q^n
=
q+15q^2+81q^3+\cdots.
\]
Since the subspace of forms in $M_5(\Gamma_0(3),\chi_3)$ with vanishing constant term at $\infty$ is
one-dimensional and $tC$ has the same leading coefficient, it follows that
\[
tC=E_{5,\chi_3,\chi_0}.
\]
\end{proof}

\begin{proposition}\label{prop:defect-only-infty}
For every $r\ge1$, the exact defect
\[
\widetilde G_r:=T_p\!\left(\frac{C}{t^{rp}}\right)-\frac{C}{t^r}
\]
is a weakly holomorphic modular form of weight $5$ on $\Gamma_0(3)$ with character $\chi_3$,
holomorphic at the cusp $0$ and with poles only at the cusp $\infty$.
\end{proposition}

\begin{proof}
The form $C/t^{rp}$ belongs to $M_5^!(\Gamma_0(3),\chi_3)$. By Proposition~\ref{prop:fricke-C},
\[
\ord_0\!\left(\frac{C}{t^{rp}}\right)=rp+1,
\qquad
\ord_0\!\left(\frac{C}{t^r}\right)=r+1,
\]
so both $C/t^{rp}$ and $C/t^r$ are holomorphic at the cusp $0$.

Since $p\nmid3$, Lemma~\ref{lem:hecke-decomp} shows that $T_p(C/t^{rp})$ is again a weakly holomorphic
modular form of weight $5$ on $\Gamma_0(3)$ with character $\chi_3$. Writing out the definition,
\[
T_p\!\left(\frac{C}{t^{rp}}\right)(\tau)
=
\chi_3(p)p^4\frac{C(p\tau)}{t(p\tau)^{rp}}
+\frac1p\sum_{k=0}^{p-1}\frac{C((\tau+k)/p)}{t((\tau+k)/p)^{rp}}.
\]
A term can become infinite only if $p\tau$ or $(\tau+k)/p$ is $\Gamma_0(3)$-equivalent to the cusp
$\infty$, because $C/t^{rp}$ has poles only there. We verify that this cannot happen at $\tau$
equivalent to the cusp $0$.

Recall that the cusps of $\Gamma_0(3)$ are represented by $\infty$ and $0$; a reduced fraction
$a/c$ with $\gcd(a,c)=1$ is $\Gamma_0(3)$-equivalent to $\infty$ precisely when $3\mid c$, and to $0$
otherwise. Suppose $\tau$ is equivalent to the cusp $0$, so we may take $\tau=a/c$ with
$\gcd(a,c)=1$ and $3\nmid c$. Then
\[
p\tau=\frac{pa}{c},
\qquad
\frac{\tau+k}{p}=\frac{a+kc}{pc}
\qquad (0\le k\le p-1).
\]
For the first fraction, $\gcd(pa,c)\mid p$ since $\gcd(a,c)=1$, so after reduction the denominator
is either $c$ or $c/p$. In either case it is coprime to $3$, because $3\nmid c$ and $p\ne3$.
For the second fraction, $\gcd(a+kc,pc)\mid p$, so the reduced denominator is either $pc$ or $c$;
again coprime to $3$. Hence both $p\tau$ and $(\tau+k)/p$ are $\Gamma_0(3)$-equivalent to the cusp
$0$, not to $\infty$. Since $C/t^{rp}$ is holomorphic at the cusp $0$, every term in the definition
of $T_p(C/t^{rp})$ remains finite at such $\tau$. Therefore $T_p(C/t^{rp})$, and hence $\widetilde G_r$,
has no poles away from the cusp $\infty$. Since both $C/t^{rp}$ and $C/t^r$ are holomorphic at the
cusp $0$, the defect $\widetilde G_r$ is holomorphic there as well.
\end{proof}

\begin{proposition}\label{prop:pole-lowering}
Let $f\in M_5^!(\Gamma_0(3),\chi_3)$ have a unique pole at the cusp $\infty$, of order at most $N$.
Then
\[
f=\beta_{-1}tC+\sum_{j=0}^{N}\alpha_j\frac{C}{t^j}
\]
for suitable coefficients $\beta_{-1},\alpha_0,\dots,\alpha_N\in\C$.
\end{proposition}

\begin{proof}
If $N=0$, then $f$ is holomorphic on $X_0(3)$ and Proposition~\ref{prop:weight5-basis} gives
$f\in\Span\{C,tC\}$.

Assume $N\ge1$. Since
\[
\frac{C(q)}{t(q)^N}=q^{-N}(1+O(q)),
\]
there is a unique constant $\lambda_N$ such that
\[
f_N:=f-\lambda_N\frac{C}{t^N}
\]
has pole order at most $N-1$ at $\infty$. Repeating the same step for $f_N$, then for the resulting form,
and so on, we obtain constants $\lambda_1,\dots,\lambda_N$ such that
\[
h:=f-\sum_{j=1}^{N}\lambda_j\frac{C}{t^j}
\]
is holomorphic at $\infty$. Every term $C/t^j$ is holomorphic at the cusp $0$ by
Proposition~\ref{prop:fricke-C}, and $f$ has no poles away from $\infty$ by hypothesis. Hence $h$ is a
holomorphic modular form of weight $5$ on $\Gamma_0(3)$ with character $\chi_3$. By
Proposition~\ref{prop:weight5-basis},
\[
h=\alpha_0 C+\beta_{-1}tC
\]
for some $\alpha_0,\beta_{-1}\in\C$. Renaming $\lambda_j$ as $\alpha_j$ for $1\le j\le N$ gives the stated
decomposition.
\end{proof}

\begin{proposition}\label{prop:rr-decomposition}
For every $r\ge1$, the exact defect $\widetilde G_r$ admits a finite expansion
\[
\widetilde G_r=\beta_{-1}tC+\sum_{j\ge0}\alpha_j\frac{C}{t^j}
\]
with coefficients $\beta_{-1},\alpha_j\in\Q$ and only finitely many nonzero $\alpha_j$.
\end{proposition}

\begin{proof}
By Proposition~\ref{prop:defect-only-infty}, the form $\widetilde G_r$ is weakly holomorphic and has poles
only at the cusp $\infty$. Proposition~\ref{prop:pole-lowering} therefore gives a decomposition of the
required shape.

Its coefficients lie in $\Q$ because the $q$-expansion of $\widetilde G_r$ is rational, the basis elements
$tC$, $C$, $C/t$, $C/t^2,\dots$ have rational $q$-expansions, and this basis is unitriangular with respect
to the order at $\infty$:
\[
tC=q+O(q^2),\qquad C=1+O(q),\qquad \frac{C}{t^j}=q^{-j}(1+O(q))
\qquad (j\ge1).
\]
Hence the coefficients are uniquely determined by the $q$-expansion.
\end{proof}

\subsection*{5.3. Fricke--Hecke intertwining}

\begin{lemma}[Fricke--Hecke intertwining]\label{lem:fricke-hecke}
Let $k\ge0$ and let $p\ge5$ be prime. On $M_k^!(\Gamma_0(3),\chi_3)$ one has
\[
T_pW_3=\chi_3(p)\,W_3T_p.
\]
Equivalently, for every $f\in M_k^!(\Gamma_0(3),\chi_3)$,
\[
T_p(f|W_3)=\chi_3(p)\,(T_pf)|W_3.
\]
\end{lemma}

\begin{remark}\label{rem:quadratic-character}
Both sides of the identity implicitly rely on the fact that $\chi_3$ is \emph{quadratic}
($\chi_3^2=\chi_0$). Conjugation by $w_3$ sends
\[
\gamma=\begin{pmatrix}a&b\\3c&d\end{pmatrix}\in\Gamma_0(3)
\quad\longmapsto\quad
w_3\gamma w_3^{-1}=\begin{pmatrix}d&-c\\-3b&a\end{pmatrix}\in\Gamma_0(3),
\]
which swaps the lower-right entries $d$ and $a$. Hence the nebentypus character of $f|_kW_3$ is
$\gamma\mapsto\chi_3(a)$ rather than $\gamma\mapsto\chi_3(d)$. Because $\chi_3(a)=\chi_3(d)$ whenever
$ad\equiv1\pmod{3}$ (which is forced by $\det\gamma=1$ and $3\mid c$), the two characters agree:
$\chi_3(a)\chi_3(d)^{-1}=\chi_3(ad)=\chi_3(1)=1$. Thus $f|_kW_3$ remains in the same nebentypus space
$M_k^!(\Gamma_0(3),\chi_3)$. For a non-quadratic character the analogous statement requires replacing
$\chi$ by $\bar\chi$ on the right-hand side, and the intertwining identity must be reformulated
accordingly.
\end{remark}

\begin{proof}
For $p\nmid3$, the Hecke operator may be written in slash form as
\[
T_pf
=
p^{k/2-1}\left(
\sum_{b=0}^{p-1}f|_k\alpha_b
+\chi_3(p)\,f|_k\beta
\right),
\]
where
\[
\alpha_b:=\begin{pmatrix}1&b\\ 0&p\end{pmatrix},
\qquad
\beta:=\begin{pmatrix}p&0\\ 0&1\end{pmatrix}.
\]
Indeed,
\[
f|_k\alpha_b=p^{-k/2}f\!\left(\frac{\tau+b}{p}\right),
\qquad
f|_k\beta=p^{k/2}f(p\tau),
\]
so after multiplying by $p^{k/2-1}$ one recovers the formula of Lemma~\ref{lem:hecke-decomp}.

We now compare the matrices $w_3\alpha_b$ and $w_3\beta$ with the same coset representatives on the other
side of~$w_3$. First,
\[
w_3\alpha_0
=\begin{pmatrix}0&-p\\ 3&0\end{pmatrix}
=\beta w_3,
\qquad
w_3\beta
=\begin{pmatrix}0&-1\\ 3p&0\end{pmatrix}
=\alpha_0 w_3.
\]
Now let $b\in\{1,\dots,p-1\}$. Choose $b'\in\{1,\dots,p-1\}$ such that
\[
3bb'\equiv -1 \pmod p.
\]
Define
\[
\gamma_b:=
\begin{pmatrix}
p & -b'\\
-3b & \dfrac{1+3bb'}{p}
\end{pmatrix}.
\]
Since $3bb'\equiv-1\pmod p$, the lower-right entry is an integer; moreover
\[
\det(\gamma_b)
=
p\cdot\frac{1+3bb'}{p}-(-b')(-3b)
=
1,
\]
and the lower-left entry is divisible by $3$, so $\gamma_b\in\Gamma_0(3)$. A direct multiplication gives
\[
\gamma_b\alpha_{b'}w_3
=
\begin{pmatrix}0&-p\\ 3&3b\end{pmatrix}
=
w_3\alpha_b.
\]
Hence
\[
w_3\alpha_b=\gamma_b\alpha_{b'}w_3
\qquad (1\le b\le p-1).
\]

Because $f\in M_k^!(\Gamma_0(3),\chi_3)$, for every
$\gamma=\begin{pmatrix}a&b\\ c&d\end{pmatrix}\in\Gamma_0(3)$ one has
\[
f|_k\gamma=\chi_3(d)\,f.
\]
For our $\gamma_b$,
\[
d_b:=\frac{1+3bb'}{p},
\qquad
d_bp-3bb'=1,
\]
so modulo $3$ we get
\[
d_bp\equiv1\pmod3.
\]
Therefore
\[
\chi_3(d_b)=\chi_3(p)^{-1}=\chi_3(p),
\]
since $\chi_3$ is quadratic.

Now compute:
\[
T_p(f|W_3)
=
p^{k/2-1}\left(
\sum_{b=0}^{p-1}f|_k w_3\alpha_b
+\chi_3(p)\,f|_k w_3\beta
\right).
\]
Using the identities above,
\[
f|_k w_3\alpha_0=f|_k\beta w_3,
\]
\begin{align*}
f|_k w_3\alpha_b
&=
f|_k\gamma_b\alpha_{b'}w_3 \\
&=
\chi_3(d_b)\,f|_k\alpha_{b'}w_3 \\
&=
\chi_3(p)\,f|_k\alpha_{b'}w_3
\qquad (1\le b\le p-1).
\end{align*}
and
\[
f|_k w_3\beta=f|_k\alpha_0 w_3.
\]
Since $b\mapsto b'$ is a permutation of $(\Z/p\Z)^\times$, we obtain
\[
T_p(f|W_3)
=
p^{k/2-1}\left(
f|_k\beta w_3
+\chi_3(p)\sum_{u=1}^{p-1}f|_k\alpha_u w_3
+\chi_3(p)f|_k\alpha_0 w_3
\right),
\]
hence
\[
T_p(f|W_3)
=
p^{k/2-1}\left(
f|_k\beta w_3
+\chi_3(p)\sum_{u=0}^{p-1}f|_k\alpha_u w_3
\right).
\]
Factoring out $\chi_3(p)$ and using $\chi_3(p)^2=1$ gives
\[
T_p(f|W_3)
=
\chi_3(p)\,
p^{k/2-1}\left(
\sum_{u=0}^{p-1}f|_k\alpha_u w_3
+\chi_3(p)\,f|_k\beta w_3
\right)
=
\chi_3(p)\,(T_pf)|W_3.
\]
This proves the lemma.
\end{proof}

\begin{lemma}\label{lem:ordinf-Tp}
Let $k\ge0$, let $p\ge5$ be prime, and let
\[
g(q)=\sum_{n\ge N}a_n q^n\in M_k^!(\Gamma_0(3),\chi_3),
\qquad N\ge 0.
\]
Then
\[
\ord_\infty(T_pg)\ge \left\lceil\frac{N}{p}\right\rceil.
\]
\end{lemma}

\begin{proof}
By Lemma~\ref{lem:hecke-decomp},
\[
T_pg=\Lambda_p(g)+\chi_3(p)p^{k-1}V_p(g).
\]
Since $N\ge 0$,
\[
\ord_\infty\bigl(\Lambda_p(g)\bigr)\ge \left\lceil\frac{N}{p}\right\rceil,
\qquad
\ord_\infty\bigl(V_p(g)\bigr)=pN\ge \left\lceil\frac{N}{p}\right\rceil,
\]
so the stated bound follows.
\end{proof}

\begin{remark}
The hypothesis $N\ge0$ is essential. For $N<0$ the bound $\ord_\infty(T_pg)\ge \lceil N/p\rceil$
is false in general, because the $V_p$-term contributes a pole of order $p|N|$, which is worse
than $\lceil N/p\rceil$. For a genuine weakly holomorphic example within $M_5^!(\Gamma_0(3),\chi_3)$,
take $g=C/t$, which has $\ord_\infty(g)=-1$. Then
\[
T_pg=\Lambda_p(g)+\chi_3(p)p^4 V_p(g),
\]
and the $V_p$-term contributes a pole of order $p$ at $\infty$, so for $p=5$
\[
\ord_\infty(T_5g)=-5<0=\lceil -1/5\rceil,
\]
violating the bound. In the application below the relevant input has $\ord_\infty=rp+1>0$, so
the lemma applies.
\end{remark}

\subsection*{5.4. Vanishing of $F_r$ modulo $p^4$ for all $r\ge1$}

The proof in this subsection is self-contained. The approach via applying $W_3$ to
$C/t^{rp}$ and combining with the Fricke--Hecke intertwining, which yields vanishing
for all $r\ge1$ uniformly and eliminates the need for a layer-by-layer reduction, was
suggested to the author by F.~Beukers.

\begin{theorem}[Vanishing of $F_r$ modulo $p^4$ for all $r\ge1$]\label{thm:fricke-vanishing}
For every prime $p\ge5$ and every integer $r\ge1$ one has
\[
F_r(q)=\Lambda_p\!\left(\frac{C(q)}{t(q)^{rp}}\right)-\frac{C(q)}{t(q)^r}\equiv0\pmod{p^4}.
\]
Equivalently,
\[
\widetilde G_r(q)\equiv0\pmod{p^4}.
\]
\end{theorem}

\begin{proof}
By Proposition~\ref{prop:rr-decomposition} there is a finite expansion
\[
\widetilde G_r=\beta_{-1}tC+\sum_{j=0}^{J}\alpha_j\frac{C}{t^j}
\]
for some $J\ge0$.

\emph{Step 1.} This is the required cusp-adapted decomposition.

\emph{Step 2.} By Proposition~\ref{prop:defect-congruence},
\[
F_r\equiv \widetilde G_r \pmod{p^4},
\]
and in fact
\[
\widetilde G_r-F_r=\chi_3(p)p^4V_p\!\left(\frac{C}{t^{rp}}\right),
\]
whose coefficients lie in $p^4\Z_{(p)}$. Thus the $q$-expansion of $\widetilde G_r$ is $p$-integral.

We record the integrality of the basis elements explicitly. Since $C(q)=3E_{5,\chi_0,\chi_3}(q)\in\Z[[q]]$
and
\[
\frac{1}{t(q)}\in q^{-1}\Z[[q]],
\qquad
H(q)=\frac{q}{t(q)}\in 1+q\Z[[q]],
\]
the basis elements
\[
tC=q+O(q^2),\qquad C=1+O(q),\qquad \frac{C}{t^j}=q^{-j}CH^j\in q^{-j}\Z[[q]]
\qquad (j\ge1)
\]
all have $\Z$-coefficients in their $q$-expansions; in particular they lie in $\Z_{(p)}$.
Combined with the $p$-integrality of $\widetilde G_r$, this shows that in the expansion
\[
\widetilde G_r=\beta_{-1}tC+\sum_{j=0}^{J}\alpha_j\frac{C}{t^j}
\]
one has $\beta_{-1},\alpha_j\in\Z_{(p)}$: indeed, the basis is lower-triangular with respect to the
principal part at $\infty$ (with diagonal $1$), so the coefficients are uniquely recovered from the
$q$-expansion of $\widetilde G_r$ by descending triangular solve over $\Z_{(p)}$.

By Proposition~\ref{prop:defect-pole-order},
\[
F_r=O(q^{-r+1}),
\]
so the coefficients of $q^{-J},q^{-J+1},\ldots,q^{-r}$ in $F_r$ all vanish. Combined with
\[
\widetilde G_r-F_r=\chi_3(p)p^4V_p\!\left(\frac{C}{t^{rp}}\right)\in p^4\Z_{(p)}[[q,q^{-1}]],
\]
this gives
\[
[q^{-n}]\widetilde G_r\equiv 0\pmod{p^4}
\qquad (r\le n\le J).
\]

We now prove $\alpha_j\equiv0\pmod{p^4}$ for all $j\ge r$ by descending induction on $j$.

If $J<r$, the claim is vacuous; in that case every $\alpha_j$ with $j\ge r$ is zero by the expansion
itself and there is nothing to prove. We assume henceforth $J\ge r$.

\emph{Base case $j=J$.} The only basis element contributing to $[q^{-J}]\widetilde G_r$ is $C/t^J$
(other elements $C/t^{j'}$ with $j'<J$ start at $q^{-j'}$ with $j'<J$, and $tC$, $C$ start at
$q^0$ or higher). Hence
\[
[q^{-J}]\widetilde G_r=\alpha_J\cdot[q^{-J}](C/t^J)=\alpha_J\cdot 1=\alpha_J.
\]
Since $J\ge r$, the left side is $\equiv0\pmod{p^4}$, so $\alpha_J\equiv0\pmod{p^4}$.

\emph{Inductive step.} Suppose $\alpha_J,\alpha_{J-1},\ldots,\alpha_{j+1}\equiv0\pmod{p^4}$ for some
$j$ with $r\le j<J$. The coefficient of $q^{-j}$ in $\widetilde G_r$ is
\[
[q^{-j}]\widetilde G_r
=
\alpha_j+\sum_{j'=j+1}^{J}\alpha_{j'}\cdot[q^{-j}](C/t^{j'}),
\]
where each $[q^{-j}](C/t^{j'})\in\Z$. By inductive hypothesis, every $\alpha_{j'}$ with $j'>j$ is
$\equiv0\pmod{p^4}$, so the sum is $\equiv0\pmod{p^4}$. Since $[q^{-j}]\widetilde G_r\equiv0\pmod{p^4}$
(because $j\ge r$), we conclude $\alpha_j\equiv0\pmod{p^4}$.

This proves $\alpha_j\equiv0\pmod{p^4}$ for all $j\ge r$.

\emph{Step 3.} Apply $W_3$ to $C/t^{rp}$. By Proposition~\ref{prop:fricke-C},
\[
\left(\frac{C}{t^{rp}}\right)\Big|W_3
=
-27i\,3^{6rp}Ct^{rp+1},
\qquad
\left(\frac{C}{t^r}\right)\Big|W_3
=
-27i\,3^{6r}Ct^{r+1}.
\]
The first form has order $rp+1$ at $\infty$, so Lemma~\ref{lem:ordinf-Tp} gives
\[
\ord_\infty\!\left(
T_p\!\left(\left(\frac{C}{t^{rp}}\right)\Big|W_3\right)
\right)
\ge
\left\lceil\frac{rp+1}{p}\right\rceil
=
r+1.
\]
By Lemma~\ref{lem:fricke-hecke},
\[
\left(T_p\!\left(\frac{C}{t^{rp}}\right)\right)\Big|W_3
=
\chi_3(p)\,
T_p\!\left(\left(\frac{C}{t^{rp}}\right)\Big|W_3\right),
\]
hence
\[
\ord_\infty\!\left(
\left(T_p\!\left(\frac{C}{t^{rp}}\right)\right)\Big|W_3
\right)\ge r+1.
\]
The second transformed term also has order $r+1$ at $\infty$, so
\[
\ord_\infty(\widetilde G_r|W_3)\ge r+1.
\]

Now
\begin{align*}
(Ct^{-j})|W_3
&=
(C|W_3)\,(t^{-j}|_0W_3) \\
&=
(-27itC)\,(3^{6j}t^j) \\
&=
-27i\,3^{6j}Ct^{j+1}
\qquad (j\ge0).
\end{align*}
and the same formula also covers $j=-1$ because
\[
(tC)|W_3=-\frac{i}{27}C=-27i\,3^{-6}C.
\]
Hence
\[
\widetilde G_r|W_3
=
-27i\left(
\beta_{-1}3^{-6}C+\sum_{j=0}^{J}\alpha_j3^{6j}Ct^{j+1}
\right).
\]
The terms
\[
C,\ Ct,\ Ct^2,\ \dots
\]
have distinct orders $0,1,2,\dots$ at $\infty$, so they are linearly independent.
Since $\ord_\infty(\widetilde G_r|W_3)\ge r+1$ is an exact bound (not a congruence),
and since the scalars $3^{-6}$ and $3^{6j}$ are nonzero, the coefficients of
$C, Ct, Ct^2, \ldots, Ct^r$ in the expansion of $\widetilde G_r|W_3$ vanish exactly.
Hence
\[
\beta_{-1}=0,
\qquad
\alpha_j=0
\qquad (0\le j\le r-1).
\]
These are \emph{exact} equalities, not merely congruences modulo $p^4$.

Combining Steps~2 and~3 gives
\[
\beta_{-1}=0,
\qquad
\alpha_j=0
\quad (0\le j\le r-1),
\qquad
\alpha_j\equiv0\pmod{p^4}
\quad (j\ge r).
\]
Therefore
\[
\widetilde G_r\equiv0\pmod{p^4}.
\]
Finally, Proposition~\ref{prop:defect-congruence} gives
\[
F_r\equiv \widetilde G_r\equiv0\pmod{p^4},
\]
as claimed.
\end{proof}

\begin{theorem}[Theorem A]\label{thm:super}
For every prime $p\ge5$ and every integer $m\ge1$,
\[
A(pm)\equiv A(m)\pmod{p^4}.
\]
\end{theorem}

\begin{proof}
By Theorem~\ref{thm:LB},
\[
B_m=\cst_q\!\left(\frac{C(q)}{t(q)^m}\right).
\]
Since $\cst_q(\Lambda_p(f))=\cst_q(f)$ for every Laurent series $f$,
\[
B_{mp}
=
\cst_q\!\left(\frac{C(q)}{t(q)^{mp}}\right)
=
\cst_q\!\left(\Lambda_p\!\left(\frac{C(q)}{t(q)^{mp}}\right)\right).
\]
Therefore
\[
B_{mp}-B_m
=
\cst_q\!\left(
\Lambda_p\!\left(\frac{C(q)}{t(q)^{mp}}\right)-\frac{C(q)}{t(q)^m}
\right)
=
\cst_q(F_m).
\]
By Theorem~\ref{thm:fricke-vanishing}, $F_m\equiv0\pmod{p^4}$, hence
\[
B_{mp}\equiv B_m\pmod{p^4}.
\]
Since $B_n=(-1)^nA_n$ and $p$ is odd, this is equivalent to
\[
A(pm)\equiv A(m)\pmod{p^4}.
\]
\end{proof}

\end{document}